\documentclass[reqno,11pt]{amsart}

\usepackage{color}

\newtheorem{theorem}{Theorem}[section]
\newtheorem{lemma}[theorem]{Lemma}
\newtheorem{proposition}[theorem]{Proposition}

\theoremstyle{definition}

\theoremstyle{remark}
\newtheorem{remark}[theorem]{Remark}

\newcommand{\mysection}[1]{\section{#1}
\setcounter{equation}{0}}

\newcommand{\bR}{\mathbb R}

\newcommand\cL{\mathcal{L}}

\def\bH{\mathbb{H}}
\def\cH{\mathcal{H}}

\newcommand{\esssup}{\operatorname*{ess\,sup}}
\renewcommand{\epsilon}{\varepsilon}

\begin{document}
\title[Regularity criteria for MHD]{Regularity criteria for suitable weak solutions to the four dimensional incompressible magneto-hydrodynamic equations near boundary}

\author[X. Gu]{Xumin Gu}
\address[X. Gu]{School of Mathematical Sciences, Fudan University,
Shanghai 200433, People's Republic of China}
\email{xumingu11@fudan.edu.cn}

\begin{abstract}
In this paper, we consider suitable weak solutions of the four dimensional incompressible magneto-hydrodynamic equations. We give two different kind $\varepsilon$-regularity criteria. One only requires the smallness of scaling $L^{p,q}$ norm of $u$, another requires the smallness of scaling space time $L^2$ norm of $\nabla u$ and boundedness of scaling norm of $H$ or $\nabla H$. And as an application of the second kind criteria, we also prove that up to the boundary, the two-dimensional Hausdorff measure of the set of singular points is equal to zero.
\end{abstract}

\maketitle

\mysection{Introduction}
In this paper we consider the four dimensional incompressible magneto-hydrodynamic (MHD) equations:
\begin{equation}
\label{ns}
\left\{\begin{aligned}
u_t +u\cdot \nabla u -\Delta u + \nabla \Pi &= H\cdot \nabla H,\\
\nabla \cdot u &= 0,\\
H_t +u\cdot \nabla H -\Delta H- H\cdot \nabla u &=0,\\
\nabla \cdot H &=0.
\end{aligned}\right.
\end{equation}
in a cylindrical domain $Q_T\equiv\Omega\times(0,T)$, where $\Omega \subset \mathbb{R}^4$ is smooth. Here $u$ is the velocity vector, $H$ is the magnetic vector and $\Pi=p+\frac{|H|^2}{2}$ is the magnetic pressure. The boundary conditions of $u$ and $H$ are given as following:

\begin{equation}
\label{bd}
u=0,\quad \text{and} \quad H\cdot \nu=0, \nabla H\cdot \nu=0, \quad \forall\,\,x\in \partial\Omega,
\end{equation}
where $\nu$ is the outward unit normal vector along the boundary $\partial\Omega$. The boundary condition for $H$ is equivalent to the slip-condition described in \cite{kang_12} in three dimension.  The MHD equations usually describe the dynamics of the interaction of moving conducting fluids with electro-magnetic fields which are frequently observed in nature and industry, e.g., plasma liquid metals, gases (see \cite{David_01, Duvant_72}).  We are interested in the partial regularity of suitable weak solutions $(u,\Pi,H)$ to \eqref{ns}  up to the boundary.

We say that a pair of functions $(u,\Pi,H)$ is a suitable weak solution to \eqref{ns} in $Q_T$ with the boundary condition \eqref{bd} if $(u,H) \in L_{\infty}(0,T;L_2(\Omega))\cap L_{2}(0,T;W^1_2(\Omega))$ and $\Pi \in L_{3/2}(Q_T)$ satisfy \eqref{ns} in the weak sense and additionally the generalized local energy inequality holds for any non-negative functions $\psi_1,\psi_2 \in C^{\infty}(\bar\Omega\times(0,T])$ and $t\in (0,T]$:
\begin{align}
\nonumber&\esssup_{0<s\leq t}\int_{\Omega}(|u(x,s)|^2\psi_1(x,s)+|H(x,s)|^2\psi_2(x,s)) \,dx\\ \nonumber&\quad\quad\quad+2\int_{Q_t}(|\nabla u|^2\psi_1+|\nabla H|^2\psi_2) \,dx\,ds\\
\nonumber&\quad\leq \int_{Q_t}|u|^2(\partial_t\psi_1+\Delta \psi_1)+(|u|^2+2\Pi)u\cdot \nabla \psi_1 \,dx\, \\\nonumber&\quad\quad-2\int_{Q_t}H\cdot\nabla u\cdot H \psi_1\,dx\,ds-2\int_{Q_t}(H\cdot u)(H\cdot \nabla \psi_1)\,dx\,ds\\\nonumber&\quad\quad+\int_{Q_t}|H|^2(\partial_t\psi_2+\Delta \psi_2)+|H|^2 u\cdot \nabla \psi_2 \,dx\, ds\\&\quad\quad-2\int_{Q_t}(H\cdot u)(H\cdot \nabla \psi_2)\,dx\,ds-2\int_{Q_t}(H\cdot\nabla H)\cdot u \psi_2\,dx\,ds .
\label{energy}
\end{align}

One of our main results is that, for any suitable weak solution $(u,\Pi,H)$ to \eqref{ns} with the boundary condition \eqref{bd}, the two dimensional space-time Hausdorff measure of the set of singular points up to the boundary is equal to zero.

It was shown in \cite{Duvant_72} that weak solutions for MHD equations exist globally in time and in the two-dimensional case weak solutions become regular. In the three-dimensional case, Sermange \cite{Sermange_83} proved that if a weak
solution pair $(u, H)$ are additionally in $L_{\infty}(0,T;W^1_2(\bR^3))$, then $(u, H)$ become regular. However,  the question of the regularity and uniqueness of weak solutions to the MHD equations is still widely open. Meanwhile, as in the incompressible Navier--Stokes equations (see \cite{CKN_82,Gus_07,Tian_99,Scheffer_76,Scheffer_77,Sch82,struwe_88,Vasseur_07,Wolf_11}), many authors have studied regularity conditions and the partial regularity of suitable weak solutions. He and Xin \cite{He_05} presented some interior regularity conditions of suitable
weak solutions in terms of the scaled mixed norm of the velocity and the magnetic
field, and more new interior regularity conditions were presented by Kang and Lee in \cite{Kang_09}.
These interior regularity conditions require that the scaled norm of
the velocity is small and the scaled norm of the magnetic field is bounded. Hence, the authors in \cite{Kang_09} proposed the question: "Can the regularity of suitable weak solutions be ensured without the assumption
that the scaled norm of the magnetic field is bounded?" Wang and Zhang gave a positive answer to this question in \cite{Wang_13}, they proved the following local $\varepsilon$-regularity criteria, when
\begin{equation*}
\sup_{0<r<r_0}r^{1-\frac{3}{p}-\frac{2}{q}}\|u\|_{L_{p,q}(Q_r(z_0))}\leq \varepsilon
\end{equation*}
is satisfied for some $\varepsilon$ near a interior point $z_0$, then $z_0$ is regular. Kang and Kim extended this kind criteria to the boundary case in \cite{kang_12}. The main idea of these works is that the terms induced by $\Pi$ and $H\cdot\nabla H$ in the local energy inequality \eqref{energy} all contain the velocity $u$, if some scaled norm of $u$ is small, then these terms can be controlled by timing small parameter on scale-invariant quantities of the pressure and using an iteration method. We also refer the reader to \cite{Cao_10,Chen_07,Chen_08,Wu_02,Wu_04,Zhou_05} and the references therein for extended results.

For the four dimensional case, the problem is more super-critical and the compactness arguments in the blowup procedure used, for instance, in the original paper \cite{CKN_82} as well as \cite{Lin_98, Lady_99} break down. In \cite{Han_12}, Han and He studied the four dimensional incompressible MHD equations' partial regularity for the interior case. However, the partial regularity to the four dimensional incompressible MHD equations for the boundary case seems still to be open. Meanwhile,  similar problems for Navier--Stokes equations were studied in \cite{Dong_12, Dong_13,Dong_13b, WW13}. The main idea in \cite{Dong_12, Dong_13,Dong_13b} is to first establish a weak decay estimate of certain scale-invariant quantities, and then successively improve this decay estimate by a bootstrap argument and the elliptic or parabolic regularity theory, thus the proof do not involve any compactness argument. Motivated by these works, we study the four dimensional incompressible MHD equations' partial regularity for the boundary case in this paper by following the main idea in \cite{Dong_12, Dong_13,Dong_13b}. Our results extend Kang's boundary regularity result from three dimension to four dimension and extend \cite{Han_12}'s interior result to the boundary case.

Now we state our main results, where we use some notation introduced at the beginning of the corresponding sections.

First, we have the following three boundary $\varepsilon$-regularity criteria.
\begin{theorem}
\label{th2}
Let $\Omega$ be a domain in $\mathbb{R}^4$. Let $(u,\Pi,H)$ be a suitable weak solution of \eqref{ns} in $Q_T$ with the boundary condition \eqref{bd}.
There is a positive number $\epsilon_0$ satisfying the following property. Assume that for a point $z_0=(x_0,t_0)$, $\omega(z_0,R)=Q^{+}(z_0,R)$ for some small $R$, and for some $\rho_0 >0$ we have
\begin{equation*}
C_u(\rho_0)+C_H(\rho_0)+D(\rho_0) \leq \epsilon_0.
\end{equation*}
Then $u$ and $H$ are regular at $z_0$.
\end{theorem}

\begin{theorem}
\label{lpq}
Let $\Omega$ be a domain in $\mathbb{R}^4$.  Let $(u,\Pi,H)$ be a suitable weak solution of \eqref{ns} in $Q_T$ with the boundary condition \eqref{bd}. Suppose that for every pair $(p,q)$ satisfying $\frac{4}{p}+\frac{2}{q}\leq 2$, $2<q\leq\infty$, and $(p,q)\neq (2,\infty)$, there is a positive number $\epsilon_0$ depending only on $p,q$ and satisfying the following property. Assume that for a point $z_0=(x_0,t_0)$, $\omega(z_0,R)=Q^{+}(z_0,R)$ for some small $R$ and the inequality
\begin{equation*}
\limsup_{r\searrow 0}F_u^{p,q}(r) \leq \epsilon_0,
\end{equation*}
holds. Then $u$ and $H$ are regular at $z_0$.
\end{theorem}

\begin{remark}
Motivated by \cite{kang_12, Wang_13}, we proved this $\epsilon$-regularity criteria, which only requires that the scaled norm of the velocity is small.
\end{remark}

\begin{theorem}
\label{mainthm}
Let $\Omega$ be a domain in $\mathbb{R}^4$.  Let $(u,\Pi,H)$ be a suitable weak solution of \eqref{ns} in $Q_T$ with the boundary condition \eqref{bd}. There is a positive number $\epsilon_0$ satisfying the following property. Assume that for a point $z_0=(x_0,t_0)$, $\omega(z_0,R)=Q^{+}(z_0,R)$ for some small $R$ and the inequality
\begin{equation*}
\limsup_{r\searrow 0}E_u(r) \leq \epsilon_0,  \quad \sup_{0<r<r_1}A_H(r)<\infty \,\,\text{for some } r_1
\end{equation*}
or
\begin{equation*}
\limsup_{r\searrow 0}E_u(r) \leq \epsilon_0,  \quad \sup_{0<r<r_1}E_H(r)<\infty \,\,\text{for some } r_1
\end{equation*}
holds. Then $u$ and $H$ are regular at $z_0$.
\end{theorem}

\begin{remark}
\label{rm1}
Theorem \ref{mainthm} implies that if $\limsup\limits_{r\searrow 0}(E_u(r)+E_H(r)) \leq \epsilon_0,$ then $u$ and $H$ are regular at $z_0$.
\end{remark}

Our next result is regarding the partial regularity of suitable weak solutions up to the boundary.

\begin{theorem}
\label{th3}
Let $\Omega$ be a domain in $\mathbb{R}^4$ with uniform $C^2$ boundary. Let $(u,\Pi,H)$ be a suitable weak solution of \eqref{ns} in $Q_T$ with the boundary condition \eqref{bd}. Then up to the boundary, the 2D Hausdorff measure of the set of singular points is equal to zero.
\end{theorem}
We explain the main steps and the main difficulties to prove the main results slightly below. Our proofs mainly follow the scheme in \cite{Dong_13,Dong_13b} mentioned before. Firstly, we carry out the estimates for some scaled norms. The main strategy is that we control these scale-invariant quantities in a smaller ball by their values in a larger ball and use these estimates to build iteration scheme later. And as a result of this step, we can show that if the assumption of Theorem \ref{lpq} or Theorem \ref{mainthm} is satisfied, then the assumption of Theorem \ref{th2} is satisfied. Compared to the Navier-Stokes equations (\cite{Dong_13,Dong_13b}), the main difference in the MHD equations is to control the pressure $\Pi$ and the nonlinear term $H\cdot\nabla H$. As we explained before, motivated by \cite{kang_12, Wang_13}, we can control the terms induced by $\Pi$ and $H\cdot\nabla H$ in the local energy inequality \eqref{energy} with the assumption of Theorem \ref{lpq}. However, Theorem \ref{lpq} can not infer the optimal partial regularity result. In order to get the optimal partial regularity, we still need to get an $\epsilon$-regularity criteria which use scaled norms of $\nabla u$. In this case, the space-time $L_3$ scaled norm of $u$ may not small, and we will encounter difficulties to control the pressure $\Pi$ without any assumption of $H$. On the other hand, we find that after adding bounded constrains on scaled norm of $H$, we can control the pressure $\Pi$ by timing a small parameter and employing an iteration argument. Thus, in this way, it is possible for us to get the $\epsilon$-regularity criteria: Theorem \ref{mainthm}. The other difficulty in showing the partial regularity up to the boundary is that the estimate of $\Pi$ now contains slow decay terms by using a decomposition of the pressure $\Pi$ introduced by Seregin \cite{Seregin_02}, which is different from the interior case and then the method in \cite{Han_12} seems not be applicable. We will use the iteration argument in \cite{Dong_13b} to conquer this difficulty.

Secondly, we will use an iteration scheme to establish an initial decay estimate. And lastly, we improve
this decay estimate by a bootstrap argument, and apply the parabolic regularity
theory to get a good estimate of the $L_{3/2}$-mean oscillations of $u$ and $H$ and prove the main results.

We remark that by using the same method we can get an alternative proof of Kang, Kim's results \cite{kang_12} without using any compactness argument and our method also provides a different approach than \cite{Han_12} to prove interior partial regularity results. It remains an interesting open problem whether a similar result can be obtained for higher dimensional MHD equations ($d\geq 5$ for the time-dependent case).  It seems to
us that four is the highest dimension to which our approach (or any
existing approach) applies. In fact, by the embedding theorem
$$
L_{\infty}((0,T);L_2(\Omega))\cap L_2((0,T);W_2^1(\Omega))\hookrightarrow L_{2(d+2)/d}((0,T)\times \Omega),
$$
which implies nonlinear term in the energy inequality cannot
be controlled by the energy norm when $d \ge 5$.

This paper is organized as follows: In Section \ref{s1}, we introduce the notation of certain scale-invariant quantities and some settings which will be used throughout the paper. In Section \ref{s2}, we prove our results in three steps. In the first step, we give some estimates of the scale-invariant quantities, which are by now standard and essentially follow the arguments in \cite{Dong_07, Lin_98}. In the second step, we establish a weak decay estimate of certain scale-invariant quantities by using an iteration argument based on the estimates we proved in the first step. In the last step, we successively improve the decay estimate by a bootstrap argument, and apply parabolic regularity to get a good estimate of $L_{3/2}$-mean oscillations of $u,H$, which yields the H\"{o}lder continuity of $u,H$ according to Campanato's characterization of H\"older continuous functions.

\mysection{Notation and Settings}\label{s1}

In this section, we introduce the notation which will be used throughout this paper.  Let $\Omega$ be a domain in $\bR^4$, $-\infty\le S<T\le \infty$, and $m,n\in [1,\infty]$. We denote $L_{m,n}(\Omega\times (S,T))$ to be the usual space-time Lebesgue spaces of functions with the norm
\begin{align*}
\|f\|_{L_{m,n}(\Omega\times (S,T))}=(\int_S^T\|f\|_{L_m(\Omega)}^n\,dt)^{\frac{1}{n}}\quad\text{for}\quad n <+\infty,\\
\|f\|_{L_{m,n}(\Omega\times (S,T))}=\esssup_{t\in(S,T)}\|f\|_{L_m(\Omega)}\quad\text{for}\quad n = +\infty.
\end{align*}
We will also use the following Sobolev spaces:
\begin{align*}
W_{m,n}^{1,0}(\Omega\times (S,T))&=\Big\{f\,\Big|\, \|f\|_{L_{m,n}(\Omega\times (S,T))}+\|\nabla f\|_{L_{m,n}(\Omega\times (S,T))} < +\infty\Big\},\\
W_{m,n}^{2,1}(\Omega\times (S,T))&=\Big\{f\,\Big|\, \|f\|_{L_{m,n}(\Omega\times (S,T))}+\|\nabla f\|_{L_{m,n}(\Omega\times (S,T))}\\&\quad+\|\nabla^2 f\|_{L_{m,n}(\Omega\times (S,T))}+\|\partial_t f\|_{L_{m,n}(\Omega\times (S,T))} < +\infty\Big\}.
\end{align*}
Let $p\in (1,\infty)$. We denote $\cH_p^1$ to be the solution spaces for divergence form parabolic equations. Precisely,
$$
\cH^1_p(\Omega\times (S,T) )=
\{u: u,Du \in L_p(\Omega\times (S,T) ),\,u_t \in \bH^{-1}_p(\Omega\times (S,T) \},
$$
where $\bH^{-1}_p(\Omega\times (S,T) )$ is the space consisting of all generalized functions $v$ satisfying
$$
\inf \big\{\|f\|_{L_p(\Omega\times (S,T)) }+\|g\|_{L_p(\Omega\times (S,T)) }\,|\,v=\nabla\cdot g+f\big\}<\infty.
$$

We shall use the following notation of spheres, balls, parabolic cylinders and so onㄩ
\begin{align*}
&B(x_0,r)=\{x\in\mathbb{R}^4\,|\,|x-x_0|<r\},\ \ B(r)=B(0,r),\ \ B=B(1);\\
&B^{+}(x_0,r)=\{x\in B(x_0,r)\,|\,x=(x',x_4),\,x_4>x_{04}\},\\&B^{+}(r)=B^{+}(0,r),\ \ B^{+}=B^{+}(1);\\
&S^+(x_0,r)=\{x\in \bR^4\,|\,|x-x_0|=r,\,x=(x',x_4),\,x_4>x_{04} \};\\
&Q(z_0,r)=B(x_0,r) \times (t_0-r^2,t_0),\ \ Q(r)=Q(0,r),\ \ Q=Q(1);\\
&Q^{+}(x_0,r)=B^{+}(x_0,r) \times (t_0-r^2,t_0),\ \ Q^{+}(r)=Q^{+}(0,r),\ \ Q^{+}=Q^{+}(1);\\
&\Omega(x_0,r)=B(x_0,r)\cap \Omega,\ \ \omega(z_0,r)=Q(z_0,r)\cap Q_T,
\end{align*}
where $z_0=(x_0,t_0)$.

We also denote mean values of summable functions as follows:
\begin{align*}
[u]_{x_0,r}(t)&=\dfrac{1}{|\Omega|}\int_{\Omega(x_0,r)}u(x,t)\,dx,\\
(u)_{z_0,r}&=\dfrac{1}{|\omega|}\int_{\omega(z_0,r)}u\, dz,
\end{align*}
where $|A|$ as usual denotes the Lebesgue measure of the set $A$.


Now we introduce the following important quantities:
\begin{align*}
A_u(r)&=A_u(r,z_0)=\esssup_{t_0-r^2\leq t\leq t_0}\dfrac{1}{r^2}\int_{\Omega(x_0,r)}|u(x,t)|^2\,dx,\\
E_u(r)&=E_u(r,z_0)=\dfrac{1}{r^2}\int_{\omega(z_0,r)}|\nabla u|^2\,dz,\\
C_u(r)&=C_u(r,z_0)=\dfrac{1}{r^3}\int_{\omega(z_0,r)}|u|^3\,dz,\\
F_u^{p,q}(r)&=F_u^{p,q}(r,z_0)=\dfrac{1}{r^{\frac{4}{p}+\frac{2}{q}-1}}\bigg[\int_{t_0-r^2}^{t_0}(\int_{\Omega(x_0,r)}|u|^p\,dx)^{\frac{q}{p}}\,dt\bigg]^{\frac{1}{q}},\\
D(r)&=D(r,z_0)=\dfrac{1}{r^3}\int_{\omega(z_0,r)}|\Pi-[\Pi]_{x_0,r}|^{\frac 32}\,dz,\\
G^{\kappa,\lambda}&=G^{\kappa,\lambda}(r,z_0)=\dfrac{1}{r^{\frac{4}{\kappa}+\frac{2}{\lambda}-2}}\bigg[\int_{t_0-r^2}^{t_0}(\int_{\Omega(z_0,r)}|\Pi-[\Pi]_{x_0,r}|^{\kappa}\,dx)^{\frac{\lambda}{\kappa}})\,dt\bigg]^{\frac{1}{\lambda}},
\end{align*}
\begin{align*}
A_H(r)&=A_H(r,z_0)=\esssup_{t_0-r^2\leq t\leq t_0}\dfrac{1}{r^2}\int_{\Omega(x_0,r)}|H(x,t)|^2\,dx,\\
F_H^{p,q}(r)&=F_H^{p,q}(r,z_0)=\dfrac{1}{r^{\frac{4}{p}+\frac{2}{q}-1}}\bigg[\int_{t_0-r^2}^{t_0}(\int_{\Omega(x_0,r)}|H|^p\,dx)^{\frac{q}{p}}\,dt\bigg]^{\frac{1}{q}},\\
E_H(r)&=E_H(r,z_0)=\dfrac{1}{r^2}\int_{\omega(z_0,r)}|\nabla H|^2\,dz,\\
C_H(r)&=C_H(r,z_0)=\dfrac{1}{r^3}\int_{\omega(z_0,r)}|H|^3\,dz.\\
\end{align*}
Notice that all these quantities are invariant under the natural scaling:
$$
u_\lambda(x,t)=\lambda u(\lambda x,\lambda^2 t), \,\,
\Pi_\lambda(x,t)=\lambda^2 \Pi(\lambda x,\lambda^2 t),\,\,
H_\lambda(x,t)=\lambda H(\lambda x,\lambda^2 t).
$$
We shall estimate them in Section \ref{s2}. 


\mysection{The proof}\label{s2}
In the proofs below, we will make use of the following well-known interpolation inequality.
\begin{lemma}
\label{interpolation}
For any functions $u \in W_2^1(\mathbb{R}^4_{+})$, $u=0$ on $x_4=0$,  and real numbers $q\in [2,4]$ and $r >0$,
\begin{multline*}
\int_{B^{+}(r)}|u|^q \,dx \leq N(q)\Big[\big(\int_{B^{+}(r)}|\nabla u|^2 \, dx\big)^{q-2}\big(\int_{B^{+}(r)}|u|^2 \, dx\big)^{2-q/2}\\
+r^{-2(q-2)}\big(\int_{B^{+}(r)}|u|^2\,dx\big)^{q/2}\Big].
\end{multline*}
\end{lemma}

Let $\cL:=\partial_t-\partial_{x_i}(a_{ij}\partial_{x_j})$ be a (possibly degenerate) divergence form parabolic operator with measurable coefficients which are bounded by a constant $K>0$. We will use the following Poincar\'e type inequality for solutions to parabolic equations. See, for instance, \cite[Lemma 3.1]{Krylov_05}.
\begin{lemma}
                    \label{lem11.31}
Let $z_0\in \bR^{d+1}$, $p\in (1,\infty)$, $r\in (0,\infty)$, $u\in \cH^1_{p}(Q^{+}(z_0,r))$, $g=(g_1,\ldots,g_d),f\in L_{p}(Q^{+}(z_0,r))$. Suppose that $u$ is a weak solution to $\cL u=\nabla\cdot g+f$ in $Q^{+}(z_0,r)$. Then we have
$$
\int_{Q^{+}(z_0,r)}|u(t,x)-(u)_{z_0,r}|^p\,dz\le Nr^p
\int_{Q^{+}(z_0,r)}\big(|\nabla u|^p+|g|^p+r^p|f|^p\big)\,dz,
$$
where $N=N(d,K,p)$.
\end{lemma}

Lastly, we recall the following two important lemmas which will be used to handle the estimate for the pressure $\Pi$.
\begin{lemma}
\label{lest_1}
Let $\Omega \subset \bR^4$ be a bounded domain with smooth boundary and $T >0$ be a constant. Let $1< m <+\infty, 1< n <+\infty$ be two fixed numbers. Assume that $g \in L_{m,n}(Q_T)$. Then there exists a unique function pair $(v,p)$, which satisfies the following equations:
$$
\left\{\begin{aligned}
\partial_t v-\Delta v+\nabla p &= g \quad \text{in}\quad Q_T,\\
\nabla \cdot v&=0\quad\text{in}\quad Q_T,\\
[p]_{\Omega}(t)&=0\quad\text{for}\quad\text{a.e.}\,\, t\in [0,T],\\
v&=0 \quad\text{on}\quad \partial_p Q_T.
\end{aligned}
\right.
$$
Moreover, $v$ and $p$ satisfy the following estimate:
\begin{align*}
\|v\|_{W_{m,n}^{2,1}(Q_T)}+\|p\|_{W_{m,n}^{1,0}(Q_T)} \leq C\|g\|_{L_{m,n}(Q_T)}.
\end{align*}
where the constant $C$ only depends on $m,n,T$, and $\Omega$.
\end{lemma}

\begin{lemma}
\label{lest_2}
Let $1< m \leq 2$, $1< n \leq 2$, and $m\leq s <+\infty$ be constants and $g\in L_{s,n}(Q^+)$.
Assume that the functions $v \in W^{1,0}_{m,n}(Q^+)$ and $p\in L_{m,n}(Q^+)$ satisfy the equationsㄩ
$$
\left\{\begin{aligned}
\partial_t v-\Delta v+\nabla p &= g \quad\text{in}\quad Q^{+},\\
\nabla \cdot v&=0\quad\text{in}\quad Q^{+},
\end{aligned}
\right.
$$
and the boundary condition
\begin{equation*}
v=0 \quad\text{on}\quad \{y \big| y=(y',0),|y'|<1\} \times [-1,0).
\end{equation*}
Then, we have $v\in W^{2,1}_{s,n}(Q^+(\frac{1}{2}))$, $p\in W^{1,0}_{s,n}(Q^+(\frac{1}{2}))$, and
\begin{align*}
&\quad\|v\|_{W_{s,n}^{2,1}(Q^{+}(\frac{1}{2}))}+\|p\|_{W_{s,n}^{1,0}(Q^{+}(\frac{1}{2}))}\\& \leq C(\|g\|_{L_{s,n}(Q^{+})}+\|v\|_{W_{m,n}^{1,0}(Q^{+})}+\|p\|_{L_{m,n}(Q^{+})}).
\end{align*}
where the constant $C$ only depends on $m$, $n$, and $s$.
\end{lemma}

We refer the reader to \cite{MaSo94} for the proof of Lemma \ref{lest_1}, and \cite{Seregin_03, Seregin_09} for the proof of Lemma \ref{lest_2}.

Now we prove the main theorems in three steps.
\subsection{Step 1.}
First, we control the quantities $A_u$, $C_u$, $A_H$, $C_H$ and $D$ in a smaller ball by their values in a larger ball under the assumption that $F_u^{p,q}$ is sufficiently small or \eqref{ehcond1} or \eqref{ehcond2} holds. Here we follow the argument in \cite{Dong_07}, which in turn used some ideas in \cite{Lady_99, Lin_98, Seregin_02}.

%
%
\begin{lemma}
                                    \label{lem2.5}
Suppose $\gamma \in (0,1)$ and $\rho >0$  are constants, and $\omega(z_0,\rho)=Q^{+}(z_0,\rho)$. Then we have
\begin{equation}
\label{C_est}
C_u(\gamma \rho) \leq N[\gamma^{-3}A_u^{1/2}(\rho)E_u(\rho)+\gamma^{-9/2}A_u^{3/4}(\rho)E_u^{3/4}(\rho)+\gamma C_u(\rho)],
\end{equation}
\begin{equation}
\label{C_est2}
C_H(\gamma \rho) \leq N[\gamma^{-3}A_H^{1/2}(\rho)E_H(\rho)+\gamma^{-9/2}A_H^{3/4}(\rho)E_H^{3/4}(\rho)+\gamma C_H(\rho)],
\end{equation}
where $N$ is a constant independent of $\gamma$, $\rho$, and $z_0$.
\end{lemma}
\begin{proof}
This is Lemma 2.8 of \cite{Dong_07} with the only difference that balls (or cylinders) are replaced by half balls (or half cylinders, respectively). By using the zero boundary condition, the proof remains the same with obvious modifications. We omit the details.
\end{proof}

\begin{lemma}
\label{G_est_lemma}
Suppose $ \gamma \in (0,1/4]$, $\rho >0, 1<\lambda<2, \frac{4}{\kappa}+\frac{2}{\lambda}\geq 4, \kappa > \frac{4}{3}$ are constants, $\kappa'$ is any large number such that $3-\frac{2}{\lambda}-\frac{4}{\kappa'}>0$, and $\omega(z_0,\rho)=Q^{+}(z_0,\rho)$. Then we have
\begin{align}
\nonumber
G^{\kappa,\lambda}(\gamma \rho) \leq &N[\gamma^{2-\frac{4}{\kappa}-\frac{2}{\lambda}}(A_u^{\frac{2-\kappa}{\kappa}}(\rho)E_u^{\frac{2\kappa-2}{\kappa}}(\rho)+A_H^{\frac{2-\kappa}{\kappa}}(\rho)E_H^{\frac{2\kappa-2}{\kappa}}(\rho))
\\&\quad+\gamma^{3-\frac{2}{\lambda}-\frac{4}{\kappa'}}(G^{\kappa,\lambda}(\rho)+A_u^{1/2}(\rho)+E_u^{1/2}(\rho))],
\label{G_est}
\end{align}
where $N$ is a constant independent of $\gamma$, $\rho$, and $z_0$. In particular, for $\kappa=\lambda=\frac{3}{2}, \kappa'=24$, we have
\begin{align}
\nonumber
D(\gamma \rho) \leq &N[\gamma^{-3}(A_u^{1/2}(\rho)E_u(\rho)+A_H^{1/2}(\rho)E_H(\rho))
\\&\quad+\gamma^{9/4}(D(\rho)+A_u^{3/4}(\rho)+E_u^{3/4}(\rho))].
\label{D_est}
\end{align}
\end{lemma}

\begin{proof}
Without loss of generality, by shifting the coordinates we may assume that $z_0=(0,0)$. By the scale-invariant property, we may also assume $\rho=1$.
We choose and fix a domain $\tilde{B} \subset \mathbb{R}^{4}$ with smooth boundary so that
\begin{equation*}
B^{+}(1/2) \subset \tilde{B} \subset B^{+},
\end{equation*}
and denote $\tilde{Q}=\tilde{B}\times(-1,0)$.
Define $\tilde{f}=- u\cdot\nabla u+H\cdot \nabla H$. For $1<\lambda<2, \frac{4}{\kappa}+\frac{2}{\lambda}\geq 4, \kappa>\frac 4 3$, then $2< \frac{4\kappa}{4-\kappa}< 4$. By using H\"{o}lder's inequality, Lemma \ref{interpolation} and the Poincar\'e inequality, we get

\begin{align}
    \label{eq5.37}
\nonumber&\quad(\int_{B^{+}}|\tilde f|^{\frac{4\kappa}{\kappa+4}}\,dx)^{\frac{\kappa+4}{4\kappa}\lambda}\\\nonumber &\leq(\int_{B^{+}}|\nabla u|^2\,dx)^{\frac{\lambda}{2}}(\int_{B^{+}}|u|^{\frac{4\kappa}{4-\kappa}}\,dx)^{\frac{4-\kappa}{4\kappa}\lambda}\\\nonumber&\quad+(\int_{B^{+}}|\nabla H|^2\,dx)^{\frac{\lambda}{2}}(\int_{B^{+}}|H|^{\frac{4\kappa}{4-\kappa}}\,dx)^{\frac{4-\kappa}{4\kappa}\lambda}\\ \nonumber &\leq (\int_{B^{+}}|\nabla u|^2\,dx)^{\frac{\lambda}{2}}(\int_{B^{+}}|\nabla u|^{2}\,dx)^{\frac{3\kappa-4}{2\kappa}\lambda}(\int_{B^{+}}|u|^{2}\,dx)^{\frac{2-\kappa}{\kappa}\lambda}\\ \nonumber &\quad +(\int_{B^{+}}|\nabla H|^2\,dx)^{\frac{\lambda}{2}}(\int_{B^{+}}|\nabla H|^{2}\,dx)^{\frac{3\kappa-4}{2\kappa}\lambda}(\int_{B^{+}}|H|^{2}\,dx)^{\frac{2-\kappa}{\kappa}\lambda}\\\nonumber&\leq (\int_{B^{+}}|\nabla u|^2\,dx)^{\frac{2\kappa-2}{\kappa}\lambda}(\int_{B^{+}}|u|^2\,dx)^{\frac{2-\kappa}{\kappa}\lambda}\\&\quad+(\int_{B^{+}}|\nabla H|^2\,dx)^{\frac{2\kappa-2}{\kappa}\lambda}(\int_{B^{+}}|H|^2\,dx)^{\frac{2-\kappa}{\kappa}\lambda}
\end{align}
Since we also have $\frac{4\kappa}{4+\kappa} >1$, then by Lemma \ref{lest_1}, there is a unique solution $v\in W^{2,1}_{\frac{4\kappa}{\kappa+4},\lambda}(\tilde{Q})$ and $p_1\in W^{1,0}_{\frac{4\kappa}{\kappa+4},\lambda}(\tilde{Q})$ to the following initial boundary value problem:
$$
\left\{\begin{aligned}
\partial_t v-\Delta v+\nabla p_1 &= \tilde{f} \quad\text{in}\quad \tilde{Q},\\
\nabla \cdot v&=0\quad\text{in}\quad \tilde{Q},\\
[p_1]_{\tilde{B}}(t)&=0\quad \text{for}\quad \text{a.e.}\,\, t\in(-1,0),\\
v&=0 \quad\text{on}\quad \partial_p  \tilde{Q}.
\end{aligned}
\right.
$$
Moreover, we have
\begin{align}
\label{v1p1}
\nonumber &\quad\|v\|_{L_{\frac{4\kappa}{\kappa+4},\lambda}(\tilde{Q})}+\|\nabla v\|_{L_{\frac{4\kappa}{\kappa+4},\lambda}(\tilde{Q})}+\|p_1\|_{L_{\frac{4\kappa}{\kappa+4},\lambda}(\tilde{Q})}+\|\nabla p_1\|_{L_{\frac{4\kappa}{\kappa+4},\lambda}(\tilde{Q})}\\ \nonumber
&\leq N\|\tilde{f}\|_{L_{\frac{4\kappa}{\kappa+4},\lambda}(\tilde{Q})}\\ \nonumber
&\le N\big(\int_{-1}^{0}
(\int_{B^{+}}|\nabla u|^2\,dx)^{\frac{2\kappa-2}{\kappa}\lambda}(\int_{B^{+}}|u|^2\,dx)^{\frac{2-\kappa}{\kappa}\lambda}\,dt
\big)^{\frac{1}{\lambda}}\\&\quad +N\big(\int_{-1}^{0}
(\int_{B^{+}}|\nabla H|^2\,dx)^{\frac{2\kappa-2}{\kappa}\lambda}(\int_{B^{+}}|H|^2\,dx)^{\frac{2-\kappa}{\kappa}\lambda}\,dt
\big)^{\frac{1}{\lambda}},
\end{align}
where in the last inequality we used \eqref{eq5.37}.

We set $w=u-v$ and $p_2=\Pi-p_1-[\Pi]_{0,1/2}$. Then $w$ and $p_2$ satisfy
$$
\left\{\begin{aligned}
\partial_t w-\Delta w +\nabla p_2 &=0\quad\text{in}\quad \tilde{Q},\\
\nabla \cdot w&=0\quad\text{in}\quad \tilde{Q},\\
w&=0 \quad\text{on}\quad \big\{\partial \tilde{B} \cap \partial\Omega\big\}\times[-1,0).
\end{aligned}
\right.
$$
By Lemma \ref{lest_2} together with a scaling and the triangle inequality, we have $p_2\in W^{1,0}_{\kappa',\lambda}(Q^{+}(1/4))$ and
\begin{align}
\nonumber&\quad \|\nabla p_2\|_{L_{\kappa',\lambda}(Q^{+}(1/4))}\\
\nonumber &\leq N\Big[\|w\|_{L_{\frac{4\kappa}{\kappa+4},\lambda}(Q^{+}(1/2))}+\|\nabla w\|_{L_{\frac{4\kappa}{\kappa+4},\lambda}(Q^{+}(1/2))}+
\|p_2\|_{L_{\frac{4\kappa}{\kappa+4},\lambda}(Q^{+}(1/2))}\Big]\\
\nonumber&\leq N\Big[\|u\|_{L_{\frac{4\kappa}{\kappa+4},\lambda}(Q^{+}(1/2))}+\|\nabla u\|_{L_{\frac{4\kappa}{\kappa+4},\lambda}(Q^{+}(1/2))}\\
\nonumber&\quad+\|\Pi-[\Pi]_{0,1/2}\|_{L_{\frac{4\kappa}{\kappa+4},\lambda}(Q^{+}(1/2))}
+\|v\|_{L_{\frac{4\kappa}{\kappa+4},\lambda}(Q^{+}(1/2))}
\\&\quad+\|\nabla v\|_{L_{\frac{4\kappa}{\kappa+4},\lambda}(Q^{+}(1/2))}+
\|p_1\|_{L_{\frac{4\kappa}{\kappa+4},\lambda}(Q^{+}(1/2))}\Big].
\label{p2}
\end{align}
Here the constant $\kappa'$ is any sufficient large number. Then with \eqref{v1p1} and H\"older's inequality, we can obtain
\begin{align}
                            \label{eq5.43}
&\quad\nonumber \|\nabla p_2\|_{{L_{\kappa',\lambda}(Q^{+}(1/4))}}\\
\nonumber&\leq  N\Big[\|u\|_{L_{\frac{4\kappa}{\kappa+4},\lambda}(Q^{+}(\frac 1 2))}+\|\nabla u\|_{L_{\frac{4\kappa}{\kappa+4},\lambda}(Q^{+}(\frac 1 2))}
+\|\Pi-[\Pi]_{0,1/2}\|_{L_{\frac{4\kappa}{\kappa+4},\lambda}(Q^{+}(\frac 1 2))}\\ \nonumber
&\quad+N\big(\int_{-1}^{0}
(\int_{B^{+}}|\nabla u|^2\,dx)^{\frac{2\kappa-2}{\kappa}\lambda}(\int_{B^{+}}|u|^2\,dx)^{\frac{2-\kappa}{\kappa}\lambda}\,dt
\big)^{\frac{1}{\lambda}}\\&\quad +N\big(\int_{-1}^{0}
(\int_{B^{+}}|\nabla H|^2\,dx)^{\frac{2\kappa-2}{\kappa}\lambda}(\int_{B^{+}}|H|^2\,dx)^{\frac{2-\kappa}{\kappa}\lambda}\,dt
\big)^{\frac{1}{\lambda}}\Big].
\end{align}
Recall that $0 < \gamma \leq 1 /4$. Then by using the Sobolev--Poincar\'e, the triangle inequality, \eqref{v1p1}, \eqref{eq5.43}, 
and H\"{o}lder's inequality, we bound $D(\gamma)$ by
\begin{align*}
&\quad \dfrac{N}{\gamma^{2-\frac{4}{\kappa}-\frac{2}{\lambda}}}\big(\int_{-\gamma^2}^{0}(\int_{B^{+}(\gamma)} |\nabla p_1|^{\frac{4\kappa}{\kappa+4}}\,dx)^{\frac {\kappa+4}{4\kappa}\lambda}+(\int_{B^{+}(\gamma)} |\nabla p_2|^{\frac{4\kappa}{\kappa+4}}\,dx)^{\frac {\kappa+4}{4\kappa}\lambda}\,dt\big)^{\frac{1}{\lambda}}\\ \nonumber&\leq N\big[\gamma^{2-\frac{4}{\kappa}-\frac{2}{\lambda}}E_u^{\frac{2\kappa-2}{\kappa}}(1)A_u^{\frac{2-\kappa}{\kappa}}(1)+\gamma^{2-\frac{4}{\kappa}-\frac{2}{\lambda}}E_H^{\frac{2\kappa-2}{2\kappa}}(1)A_H^{\frac{2-\kappa}{\kappa}}(1) \big]
\\&\quad+ N\gamma^{3-\frac{2}{\lambda}-\frac{4}{\kappa'} }\big(\int_{-\gamma^2}^{0}(\int_{B^{+}(\gamma)} |\nabla p_2|^{\kappa'}\,dx)^{\frac {\lambda} {\kappa'}}\,dt\big)^{\frac{1}{\lambda}}\\\nonumber&\leq N\big[\gamma^{2-\frac{4}{\kappa}-\frac{2}{\lambda}}E_u^{\frac{2\kappa-2}{\kappa}}(1)A_u^{\frac{2-\kappa}{\kappa}}(1)+\gamma^{2-\frac{4}{\kappa}-\frac{2}{\lambda}}E_H^{\frac{2\kappa-2}{2\kappa}}(1)A_H^{\frac{2-\kappa}{\kappa}}(1) \big]
\\&\quad+ N\gamma^{3-\frac{2}{\lambda}-\frac{4}{\kappa'}}[E_u^{\frac{2\kappa-2}{2\kappa}}(1)A_u^{\frac{2-\kappa}{\kappa}}(1)+E_H^{\frac{2\kappa-2}{2\kappa}}(1)A_H^{\frac{2-\kappa}{\kappa}}(1)+G_u^{\kappa,\lambda}(1)+A^{\frac 1 2}(1)+E^{\frac 1 2}(1)]
\\\nonumber&\leq N\big[\gamma^{2-\frac{4}{\kappa}-\frac{2}{\lambda}}[E_u^{\frac{2\kappa-2}{2\kappa}}(1)A_u^{\frac{2-\kappa}{\kappa}}(1)+E_H^{\frac{2\kappa-2}{2\kappa}}(1)A_H^{\frac{2-\kappa}{\kappa}}(1)]+\gamma^{3-\frac{2}{\lambda}-\frac{4}{\kappa'}}G_u^{\kappa',\lambda}(1)\\
&\quad+\gamma^{3-\frac{2}{\lambda}-\frac{4}{\kappa'}}[A^{\frac 1 2}(1)+E^{\frac 1 2}(1)]
\big].
\end{align*}
The lemma is proved.
\end{proof}

\begin{lemma}
                \label{aelemma}
Suppose $\gamma \in (0,1/2]$ and $\rho >0$ are constants, and $\omega(z_0,\rho)=Q^{+}(z_0,\rho)$. Then we have
\begin{align*}
A_u(\gamma &\rho)+E_u(\gamma \rho)+A_H(\gamma \rho)+E_H(\gamma \rho)\\&\leq N \gamma^{-2}\big[C_u^{2/3}(\rho)+C_H^{2/3}(\rho)+C_u(\rho)+C_u^{1/3}(\rho)D^{2/3}(\rho)+C_u^{1/3}(\rho)C_H^{2/3}(\rho)\big].
\end{align*}
In particular, when $\gamma=1/2$ we have
\begin{align}
\nonumber A_u(&\rho/2)+E_u(\rho/2)+A_H(\rho/2)+E_H(\rho/2)\\& \leq N \big[C_u^{2/3}(\rho)+C_H^{2/3}(\rho)+C_u(\rho)+C_u^{1/3}(\rho)D^{2/3}(\rho)+C_u^{1/3}(\rho)C_H^{2/3}(\rho)\big].
\label{AE_est_2}
\end{align}
\end{lemma}

\begin{proof}
As before, we assume $\rho=1$. In the energy inequality \eqref{energy}, we set $t=t_0$ and choose a suitable smooth cut-off function $\psi$ such that
\begin{align*}
\psi \equiv 0 \ \ \text{in} \ \ Q_{t_0}\setminus Q(z_0,1),\ \  0\leq \psi \leq 1 \ \ \text{in} \ \ Q_T,\\
\psi \equiv 1 \ \ \text{in} \ \ Q(z_0,\gamma),  \ \ |\partial_t \psi|+|\nabla \psi| +|\nabla^2 \psi| \leq N\ \ \text{in} \ \ Q_{t_0}.
\end{align*}
By using \eqref{energy} with $\psi_1=\psi_2=\psi$, and because $u$ is divergence free, we get
\begin{multline*}
A_u(\gamma)+2E_u(\gamma)+A_H(\gamma)+2E_H(\gamma)\\ \leq \dfrac{N}{\gamma^2}\Big[\int_{Q^{+}(z_0,1)}|u|^2+|H|^2\,dz +\int_{Q^{+}(z_0,1)}(|u|^2+|H|^2+|\Pi-[\Pi]_{x_0,1}|)|u|\,dz\Big].
\end{multline*}
Using H\"{o}lder's inequality and Young's inequality, one can obtain
\begin{equation*}
\int_{Q^{+}(z_0,1)}|u|^2\, dz \leq\big(\int_{Q^{+}(z_0,1)}|u|^3\, dz\big)^{2/3}\big(\int_{Q^{+}(z_0,1)}\,dz\big)^{1/3} \leq NC_u^{2/3}(1),
\end{equation*}
\begin{equation*}
\int_{Q^{+}(z_0,1)}|H|^2\, dz \leq\big(\int_{Q^{+}(z_0,1)}|H|^3\, dz\big)^{2/3}\big(\int_{Q^{+}(z_0,1)}\,dz\big)^{1/3} \leq NC_H^{2/3}(1),
\end{equation*}

\begin{align*}
&\quad\int_{Q^{+}(z_0,1)}|H^2||u| \,dz\\
&\leq \big(\int_{Q^{+}(z_0,1)}|H|^{3}\,dz\big)^{2/3}
\big(\int_{Q^{+}(z_0,1)}|u|^3\,dz\big)^{1/3}\\
&= C_H^{2/3}(1)C_u^{1/3}(1),
\end{align*}
\begin{align*}
&\quad\int_{Q^{+}(z_0,1)}|\Pi-[\Pi]_{x_0,1}||u| \,dz\\
&\leq \big(\int_{Q^{+}(z_0,1)}|\Pi-[\Pi]_{x_0,1}|^{3/2}\,dz\big)^{2/3}
\big(\int_{Q^{+}(z_0,1)}|u|^3\,dz\big)^{1/3}\\&= C_u^{1/3}(1)D^{2/3}(1),
\end{align*}

Then the conclusion follows immediately.
\end{proof}

\begin{lemma}
\label{eeeeee}
Suppose $\rho >0$ and $\gamma \in (0,1/8]$ are constants, $\kappa,\lambda,p$ and $q$ satisfy that $\frac{1}{\kappa}+\frac{1}{p}=1,\,\frac{1}{\lambda}+\frac{1}{q}=1$, $1<\lambda<2,\, \frac{4}{p}+\frac{2}{q}\leq 2$,  and $\omega(z_0,\rho)=Q^{+}(z_0,\rho)$. Then we have
\begin{align}
\label{ereeeeeeeeeeee}
\nonumber A_u(\gamma \rho)+E_u(\gamma \rho)+A_H&(\gamma\rho)+E_H(\gamma\rho)\leq N \gamma^2(A_u(\rho)+A_H(\rho)) \\ &+N\gamma^{-3}F_u^{p,q}(\rho)\big[[F_u^{2\kappa,2\lambda}(\rho)]^2+G^{\kappa,\lambda}(\rho)+[F_H^{2\kappa,2\lambda}(\rho)]^2\big],
\end{align}
where $N$ is a constant independent of $\rho$, $\gamma$, and $z_0$.
In particular, for $p=3,q=3$, we have
\begin{multline}
\label{ers}
A_u(\gamma \rho)+E_u(\gamma \rho)+A_H(\gamma\rho)+E_H(\gamma\rho)\leq N \gamma^2(A_u(\rho)+A_H(\rho)) \\ +N\gamma^{-3}\big[C_u(\rho)+C_u^{1/3}(\rho)D^{2/3}(\rho)+C_u^{1/3}(\rho)C_H^{2/3}(\rho)\big].
\end{multline}
\end{lemma}
\begin{proof} As before, we assume $\rho=1$. Define the backward heat kernel as
\begin{equation*}
\Gamma(x,t)=\dfrac{1}{4\pi^2(\gamma^2+t_0-t)^2}e^{-\frac{|x-x_0|^2}{2(\gamma^2+t_0-t)}}.
\end{equation*}
In the energy inequality \eqref{energy} we put $t=t_0$ and choose $\psi_1=\psi_2=\Gamma\phi$, where $\phi\in C_0^\infty( B(x_0,1)\times (t_0-1,t_0+1))$ is a suitable smooth cut-off functions satisfying
\begin{align}
\nonumber 0\leq \phi\leq 1 \ \ \text{in} \ \ \mathbb{R}^4\times \bR, \ \ \phi\equiv1 \ \ \text{in} \ \ Q(z_0,1/2),\\
|\nabla\phi| \leq N, \ \ |\nabla^2 \phi| \leq N \ \ |\partial_t\phi|\leq N\quad \text{in} \ \ \mathbb{R}^4\times\bR. \label{qq}
\end{align}
By using the equality
\begin{equation*}
\Delta \Gamma + \Gamma_t =0,
\end{equation*}
we have
\begin{align} \nonumber&\quad\int_{B^{+}(x_0,1)}(|u(x,t)|^2+|H(x,t)|^2)\Gamma(t,x)\phi(x,t) \,dx +2 \int_{Q^{+}(z_0,1)}(|\nabla u|^2+|\nabla H|^2) \Gamma \phi \,dz\\ \nonumber&\leq \int_{Q^{+}(z_0,1)}\big\{(|u|^2+|H|^2)(\Gamma \phi_t+\Gamma \Delta \phi+2\nabla\phi\nabla\Gamma)\\\nonumber &\quad +(|u|^2+2|\Pi-[\Pi]_{x_0,1}|+|H|^2)u\cdot (\Gamma\nabla \phi+\phi\nabla \Gamma)\big\}\,dz \\&\quad+2\int_{Q^{+}(z_0,1)}|H|^2|u||\nabla\Gamma\phi+\Gamma\nabla\phi|\,dz. \label{ggg}
\end{align}
With straightforward computations, it is easy to see the following three properties:\\
(i) For some constant $c>0$, on $\bar{Q^{+}}(z_0,\gamma)$ it holds that \begin{equation*}
\Gamma \phi = \Gamma \geq c \gamma^{-4}.
\end{equation*}
(ii) For any $z \in Q^{+}(z_0,1)$, we have \begin{equation*} |\Gamma(z)\phi(z)| \leq N \gamma^{-4},\ \ |\phi(z)\nabla\Gamma(z)|+|\nabla\phi(z)\Gamma(z)| \leq N \gamma^{-5}. \end{equation*}
(iii) For any $z \in Q^{+}(z_0,1)\setminus Q^{+}(z_0,\gamma)$, we have \begin{equation*} |\Gamma(z)\phi_t(z)|+|\Gamma(z)\Delta\phi(z)|+|\nabla\phi\nabla\Gamma| \leq N.
\end{equation*}
Then these properties together with \eqref{qq}, \eqref{ggg} and H\"older's inequality yield
\begin{align*}
A_u(\gamma)+E_u(\gamma)+&A_H(\gamma)+E_H(\gamma) \leq N\big[\gamma^2(A_u(1)+E_u(1))\\&+\gamma^{-3}F_u^{p,q}(1)([F_u^{2\kappa,2\lambda}(1)]^2+G^{\kappa,\lambda}(1)+[F_H^{2\kappa,2\lambda}(1)]^2))\big]. \end{align*}
Thus, the lemma is proved.
\end{proof}

\begin{proposition}
\label{prop3}
Suppose that for every pair $(p,q)$ satisfying $\frac{4}{p}+\frac{2}{q}\leq 2$, $2<q\leq\infty$, 
 then for any $\epsilon_0 >0$, there exists $\epsilon_1 >0$ depending on $p,q$ small such that the following is true. For any $z_0=(x_0,t_0)$ satisfying $\omega(z_0,R)=Q^{+}(z_0,R)$ for some small $R$ and
\begin{equation}
\limsup_{r\searrow 0} F^{p,q}_u(r) \leq \epsilon_1,
\label{econd0}
\end{equation}
we can find $\rho_0$ sufficiently small such that
\begin{equation*}
A_u(\rho_0)+E_u(\rho_0)+A_H(\rho_0)+E_H(\rho_0)+C_u(\rho_0)+C_H(\rho_0)+D(\rho_0) \leq \epsilon_0.
\label{e_conl}
\end{equation*}
\end{proposition}
\begin{proof}
By \eqref{ereeeeeeeeeeee} and \eqref{econd0}, for some small $\rho$, one can obtain
\begin{multline*}
A_u(\gamma \rho)+E_u(\gamma \rho)+A_H(\gamma\rho)+E_H(\gamma\rho)\leq N \gamma^2(A_u(\rho)+A_H(\rho)) \\ +N\gamma^{-3}\epsilon_1\big[[F_u^{2\kappa,2\lambda}(\rho)]^2+G^{\kappa,\lambda}(\rho)+[F_H^{2\kappa,2\lambda}(\rho)]^2\big],
\end{multline*}
Since $\frac{4}{p}+\frac{2}{q}\leq 2, q>2$, then $2 \leq 2\kappa <4,\frac{2\kappa-2}{\kappa}\lambda\leq 1$, and using the interpolation inequality \eqref{interpolation} and Young's inequality, we can bound $[F_u^{2\kappa,2\lambda}(\rho)]^2$ by
\begin{align*}
&\quad N\rho^{2-\frac{4}{\kappa}-\frac{2}{\lambda}}\big(\int_{t_0-\rho^2}^{t_0}\rho^{-\frac{4\kappa-4}{\kappa}\lambda}(\int_{B^{+}_{\rho}}|u|^2\,dx)^{\lambda}\,dt\big)^{\frac{1}{\lambda}}\\
&\quad\quad+N\rho^{2-\frac{4}{\kappa}-\frac{2}{\lambda}}\big(\int_{t_0-\rho^2}^{t_0}(\int_{B^{+}_{\rho}}|\nabla u|^2\,dx)^{\frac{2\kappa-2}{\kappa}\lambda}(\int_{B^{+}_{\rho}}|u|^2\,dx)^{\frac{2-\kappa}{\kappa}\lambda}\,dt\big)^{\frac{1}{\lambda}}\\
&\leq N A_u(\rho)+N E_u^{\frac{2\kappa-2}{\kappa}}(\rho)A_u^{\frac{2-\kappa}{\kappa}}(\rho)\\
&\leq N(A_u(\rho)+E_u(\rho)).
\end{align*}
Similarly, $[F_H^{2\kappa,2\lambda}(\rho)]^2$ can be bounded by $N(A_H(\rho)+E_H(\rho))$.
Then, combining with \eqref{G_est} and using Young's inequality again, we have
\begin{multline}
\label{ppppp}
A_u(\gamma \rho)+E_u(\gamma \rho)+A_H(\gamma\rho)+E_H(\gamma\rho)+\epsilon'G^{\kappa,\lambda}(\gamma\rho)\\\leq N \gamma^2(A_u(\rho)+A_H(\rho)) +N\gamma^{-3}\epsilon_1(A_u(\rho)+E_u(\rho)+A_H(\rho)+E_H(\rho))\\+N\gamma^{-3}\epsilon_1G^{\kappa,\lambda}(\rho)+
N\epsilon'\gamma^{2-\frac{4}{\kappa}-\frac{2}{\lambda}}(A_u(\rho)+E_u(\rho)+A_H(\rho)+E_H(\rho))
\\\quad+N\epsilon'\gamma^{3-\frac{2}{\lambda}-\frac{4}{\kappa'}}G^{\kappa,\lambda}(\rho)+N\epsilon'\gamma^{3-\frac{2}{\lambda}-\frac{4}{\kappa'}}(A_u(\rho)+E_u(\rho)))
+N\epsilon'\gamma^{3-\frac{2}{\lambda}-\frac{4}{\kappa'}}.
\end{multline}
For any small $\epsilon_0'>0$, we choose sufficiently small $\gamma, \epsilon',\epsilon_1$ such that
\begin{align*}
N\gamma^2+N\gamma^{3-\frac{2}{\lambda}-\frac{4}{\kappa'}}<\frac{1}{8}\\
N\epsilon'\gamma^{2-\frac{4}{\kappa}-\frac{2}{\lambda}}<\frac{1}{16},\quad
N\epsilon'\gamma^{3-\frac{2}{\lambda}-\frac{4}{\kappa'}}<\frac{1}{16}\epsilon_0',\\
N\gamma^{-3}\epsilon_1<\frac{1}{16},\quad N\gamma^{-3}\epsilon_1<\frac{1}{16}\epsilon',
\end{align*}
Then by using \eqref{ppppp} with a standard iteration argument, we can get
\begin{equation*}
A_u(\rho_1)+E_u(\rho_1)+A_H(\rho_1)+E_H(\rho_1)+\epsilon'G^{\kappa,\lambda}(\rho_1) \leq \epsilon_0'
\end{equation*}
for some $\rho_1$ small enough.

Owing to Lemma \ref{interpolation} with $q=3$, we get
\begin{equation}
C_u(\rho)\leq N\big[A_u(\rho)+E_u(\rho)\big]^{3/2},\quad C_H(\rho)\leq N\big[A_H(\rho)+E_H(\rho)\big]^{3/2},
\label{hhh}
\end{equation}
which imply that
\begin{align*}
C_u(\rho_1)+C_H(\rho_1)\leq N\epsilon_0'^{\frac{3}{2}}.
\end{align*}
To estimate $D(\rho)$, recall \eqref{D_est}, we have
\begin{equation*}
D(\gamma\rho_1)\leq N\gamma^{9/4}D(\rho_1)+N\gamma^{-3}(A_u(\rho_1)+E_u(\rho_1)+A_H(\rho_1)+E_H(\rho_1)),
\end{equation*}
then we can use the standard iteration argument again to obtain
\begin{equation*}
D(\gamma^K\rho_1)\leq 2N\gamma^{-3}\epsilon_0'
\end{equation*}
with sufficiently large $K$.

Lastly, for any small $\epsilon_0$, since the choice of $\gamma$ is independent on $\epsilon_0'$, we choose sufficiently small $\epsilon_0'$ such that $N\epsilon_0'^{3/2}\leq\epsilon_0$, $2N\gamma^{-3}\epsilon_0'\leq\epsilon_0$, and set $\rho_0=\gamma^K\rho_1$ to complete the proof.
\end{proof}

\begin{proposition}
\label{prop1}
For any $\epsilon_0 >0$, there exists $\epsilon_1 >0$ small such that the following is true. For any $z_0=(x_0,t_0)$ satisfying $\omega(z_0,R)=Q^{+}(z_0,R)$ for some small $R$ and
\begin{equation}
\limsup_{r \searrow 0} E_u(r)\leq \epsilon_1, \quad \sup_{0<r<r_1}A_H(r)\leq M<\infty \,\,\text{for some } r_1
\label{ehcond2}
\end{equation}
or
\begin{equation}
\limsup_{r\searrow 0}E_u(r) \leq \epsilon_1,  \quad \sup_{0<r<r_1}E_H(r)\leq M <\infty \,\,\text{for some } r_1
\label{ehcond1}
\end{equation}
we can find $\rho_0$ sufficiently small such that
\begin{equation*}
A_u(\rho_0)+E_u(\rho_0)+C_u(\rho_0)+A_H(\rho_0)+E_H(\rho_0)+C_H(\rho_0)+D(\rho_0) \leq \epsilon_0.
\end{equation*}
\end{proposition}
\begin{proof}

First, we derive the following inequality,
\begin{multline}
\label{ers2}
A_u(\gamma \rho)+E_u(\gamma \rho)+A_H(\gamma\rho)+E_H(\gamma\rho)\leq N \gamma^2(A_u(\rho)+A_H(\rho)) \\ +N\gamma^{-3}\big[C_u(\rho)+C_u^{1/3}(\rho)D^{2/3}(\rho)+F_u^{4,2}(\rho)F_H^{8/3,4}(\rho)\big],
\end{multline}
which is slightly different from \eqref{ers}.

To derive the above inequality, we use the same cut-off function $\psi_1=\psi_2=\Gamma\phi$ used in the proof of Lemma \ref{eeeeee} and estimate the term $\int_{Q^{+}(z_0,1)}|H|^2|u||\nabla\Gamma\phi+\Gamma\nabla\phi|\,dz$ in \eqref{ggg} as:
\begin{equation*}
\int_{Q^{+}(z_0,1)}|H|^2|u||\nabla\Gamma\phi+\Gamma\nabla\phi|\,dz\leq \gamma^{-5}F_u^{4,2}(1)F_H^{8/3,4}(1).
\end{equation*}
Then by \eqref{ggg} and H\"{o}lder's inequality, we can get
\begin{align*}
\gamma^{-2}[A_u(\gamma)+E_u(\gamma)+&A_H(\gamma)+E_H(\gamma)] \leq N\big[(A_u(1)+E_u(1))\\&+\gamma^{-5}C_u(1)\gamma^{-5}C_u^{1/3}(1)D^{2/3}(1)+\gamma^{-5}F_u^{4,2}(1)F_H^{8/3,4}(1)], \end{align*}
and \eqref{ers2} is proved by scaling.

By the interpolation inequality \eqref{interpolation} and the boundary Poinc\'are inequality, $F_u^{4,2}(\rho)$ can be bounded by
\begin{equation*}
 N\rho^{-1}\big(\int_{t_0-\rho^2}^{t_0}(\rho^{-2}\int_{B^{+}_{\rho}}|u|^2\,dx+\int_{B^{+}_{\rho}}|\nabla u|^2\,dx)\,dt\big)^{\frac{1}{2}}\leq N E_u(\rho).
 \end{equation*}
Using the interpolation inequality \eqref{interpolation} again, we also have
 \begin{align*}
 &F_{H}^{8/3,4}(\rho)\leq N(A_H(\rho)+E_H(\rho)),\\
 &C_u(\rho)\leq N A_u^{1/2}(\rho)(F_u^{4,2}(\rho)+E_u(\rho))\leq N A_u^{1/2}(\rho)E_u(\rho).
 \end{align*}


If the condition \eqref{ehcond2} is satisfied, then for some small $\rho$, we can use \eqref{ers2}, \eqref{D_est} and Young's inequality to obtain
\begin{align}
\nonumber &\quad A_u(\gamma \rho)+E_u(\gamma \rho)+A_H(\gamma\rho)+E_H(\gamma\rho)+\epsilon'D(\gamma\rho)\\
\nonumber & \leq N \gamma^2(A_u(\rho)+A_H(\rho))+ N\gamma^{-3}\big[A_u^{1/2}(\rho)E_u(\rho)+A_u^{1/6}(\rho)E_u^{1/3}(\rho)D^{2/3}(\rho)\big]\\
\nonumber &\quad +N\gamma^{-3}E_u(\rho)(A_H(\rho)+E_H(\rho)) + N\epsilon'\gamma^{-3}(A_u^{1/2}(\rho)E_u(\rho)+A_H^{1/2}(\rho)E_H(\rho))\\
\nonumber &\quad +N \epsilon'\gamma^{9/4}(D(\rho)+A_u^{3/4}(\rho)+E_u^{3/4}(\rho))\\
\nonumber &\leq N\gamma^2(A_u(\rho)+A_H(\rho))+N\epsilon_1A_u(\rho)+N\gamma^{-6}\epsilon_1
+N\epsilon_1^{1/6}D(\rho)+N\epsilon_1^{2/3}A_u(\rho)
\\\nonumber &\quad +N\epsilon_1^{2/3}\gamma^{-3}+N\gamma^{-3}\epsilon_1(A_H(\rho)+E_H(\rho))+N\epsilon'\epsilon_1A_u(\rho)+N\epsilon'\epsilon_1\gamma^{-6}
\\&\quad +N\epsilon'\gamma^{3}ME_H(\rho)+N \epsilon'\gamma^{9/4}D(\rho)+N\gamma\epsilon'^4.
\label{4au}
\end{align}
For any small $\epsilon_0$, choose sufficiently small $\gamma, \epsilon', \epsilon_1$ such that
\begin{align*}
N\gamma^2+N\gamma^{9/4} \leq \frac{1}{16},\\
N\epsilon'\gamma^3M\leq\frac{1}{16},\quad N\gamma\epsilon'^4<\frac{1}{4}\epsilon_0,\\
N\epsilon_1^{1/6}<\frac{1}{4}\epsilon',\quad N\gamma^{-6}\epsilon_1\epsilon'+N\gamma^{-3}\epsilon_1^{3/2}<\frac{1}{4}\epsilon_0,\\
N\gamma^{-3}\epsilon_1+N\epsilon_1^{2/3}+N\epsilon'\epsilon_1+N\epsilon_1\leq \frac{1}{16},
\end{align*}
then with a similar argument in the proof of Proposition \ref{prop3}, we can get the conclusion.

When \eqref{ehcond1} is satisfied, then in the last line of \eqref{4au}, the term $N\epsilon'\gamma^{3}ME_H(\rho)$ becomes $N\epsilon'\gamma^{3}MA_H(\rho)$, and the proof is almost the same.

\end{proof}

In the rest of paper, we use $A(\rho)=A_u(\rho)+A_H(\rho), C(\rho)=C_u(\rho)+C_H(\rho), E(\rho)=E_u(\rho)+E_H(\rho)$ to simplify the notation.

\subsection{Step 2.}
In the second step, we will study the decay property of $A$, $C$, $E$, and $D$ as the radius $\rho$ goes to zero.
\begin{proposition}
\label{prop2}
There exists $\epsilon_0>0$ satisfying the following property. Suppose that for some $z_0=(x_0,t_0)$, $\rho_0 >0$, and $\omega(z_0,\rho_0)=Q^{+}(z_0,\rho_0)$ satisfying
\begin{equation}
C(\rho_0)+D(\rho_0)\leq \epsilon_0.
\label{cond_cdf}
\end{equation}
Then we can find $N>0$ and $\alpha_0 \in (0,1)$ such that for any $\rho \in (0,\rho_0/4)$ and $z_1 \in Q(z_0,\rho_0/4)\cap (\partial\Omega\times(t_0-\rho_0^2/16,t_0))$, the following inequality will hold uniformly
\begin{align}
A(\rho,z_1)+C^{2/3}(\rho,z_1)+E(\rho,z_1)+ D(\rho,z_1)\leq N \rho ^{\alpha_0},
\label{cdf}
\end{align}
where $N$ is a positive constant independent of $\rho$ and $z_1$.
\end{proposition}
\begin{proof}
Let $\epsilon'>0$ be a small constant to be specified later. Due to \eqref{AE_est_2} and \eqref{cond_cdf}, one can find $\epsilon_0=\epsilon_0(\epsilon')>0$ sufficiently small such that,
\begin{equation*}
A(\rho_0/2)+E(\rho_0/2)\leq \epsilon',\quad D(\rho_0/2) \leq \epsilon'.
\end{equation*}
Without loss of generality, we can assume that $\rho_0=\epsilon'$. If $\rho_0\neq \epsilon'$, since $C$, and $D$  are invariant under the natural scaling, we can get \eqref{cdf} with $N$ proportional to $\rho_0^{-\alpha_0}$ after a scaling.

By \eqref{hhh}, we have
\begin{equation*}
C(\rho_0/2)\leq N\epsilon'^{3/2}.
\end{equation*}

For any $z_1 \in Q(z_0,\rho_0/4)\cap (\partial\Omega\times(t_0-\rho_0^2/16,t_0))$, by using
\begin{equation*}
Q^{+}(z_1,\rho_0/4)\subset Q^{+}(z_0,\rho_0/2)\subset Q_T,
\end{equation*}
we can get
\begin{equation*}
A(\rho_1,z_1)+E(\rho_1,z_1)+C^{2/3}(\rho_1,z_1)+D(\rho_1,z_1) \leq N\epsilon'
\end{equation*}
with $\rho_1=\rho_0/4$.

Next, we shall prove inductively that
\begin{equation}
A(\rho_k,z_1)+E(\rho_k,z_1)+C^{2/3}(\rho_k,z_1)\le \rho_k^{\frac{1}{10}},
\quad D(\rho_k,z_1) \leq \rho_k^{\frac{1}{10}},
\label{pop}
\end{equation} where $\rho_k=\rho_1^{(1+\beta)^k}$ and $\beta=\frac{1}{200}$ for $k=1,2,\cdots$.

It is easy to see that \eqref{pop} holds for $k=1,2,3$ by choosing $\epsilon'$ sufficiently small.
Suppose that \eqref{pop} holds for $k\ge 3$. Since $\rho_{k+1}=\rho_{k}^{1+\beta}$, by using \eqref{ers} with $\gamma=\rho_k^\beta$ and $\rho=\rho_k$, we have

\begin{align}
\nonumber A(\rho_{k+1})+E(\rho_{k+1}) &\leq N \rho_{k}^{2\beta}A(\rho_k)+N \rho_k^{-3\beta}(C(\rho_k)+C^{1/3}(\rho_k)D^{2/3}(\rho_k))\\&\leq N \rho_k^{2\beta+\frac{1}{10}}+N\rho_k^{-3\beta+\frac{7}{60}}.\label{ind1}
\end{align}

Since
$$
\min\{2\beta+\frac{1}{10},-3\beta+\frac{7}{60}\}>\frac{1}{10}(1+\beta),
$$ we have
\begin{align*}
A(\rho_{k+1})+E(\rho_{k+1})\leq N\rho_{k+1}^{\frac{1}{10}+\xi}\,\, \text{for some}\,\, \xi>0,
\end{align*}
and by \eqref{hhh},
\begin{equation*}
C(\rho_{k+1})\leq N\rho_{k+1}^{\frac{3}{20}+\frac{3}{2}\xi}.
\end{equation*}
Here $N$ is a constant independent of $k$ and $\xi$.
By taking $\epsilon'$ sufficiently small that $N\epsilon'^{\xi} < 1$, we can obtain
\begin{align}
\label{ind2}
A(\rho_{k+1})+E(\rho_{k+1})+C^{2/3}(\rho_{k+1})\leq \rho_{k+1}^{\frac{1}{10}}.
\end{align}

To estimate the remaining term $D(\rho_{k+1})$,
we apply \eqref{D_est}. It turns out that, different from above, using the estimates of $A(\rho_{k})$, $E(\rho_k)$, and $D(\rho_k)$, one cannot get the estimate of $D(\rho_{k+1})$ as claimed. Instead, we shall bound $D(\rho_{k+1})$ by using the estimates which we get in the $k-2$-th step. By defining $\tilde \beta=(1+\beta)^3-1$ and using \eqref{D_est} with $\gamma=\rho_{k-1}^{\tilde \beta}$ and $\rho=\rho_{k-2}$, we can get
\begin{align*}
D(\rho_{k+1})&\leq N \rho_{k-2}^{-3{\tilde \beta}}E_u(\rho_{k-2})A_u^{1/2}(\rho_{k-2})+N\rho_{k-2}^{-3{\tilde \beta}}E_H(\rho_{k-2})A_H^{1/2}(\rho_{k-2})\\&\quad+ N\rho_{k-2}^{9\tilde \beta/4}D(\rho_{k-2})+N\rho_{k-2}^{9\tilde \beta/4}(A_u^{3/4}(\rho_{k-2})+E_u^{3/4}(\rho_{k-2}))\\&\leq N\rho_{k-2}^{-3\tilde \beta+\frac{3}{20}}+N\rho_{k-2}^{\frac{9}{4}\tilde \beta+\frac{1}{10}}+N\rho_{k-2}^{\frac{9}{4}\tilde \beta+\frac{3}{40}}.
\end{align*}
Since $\min\{-3\tilde \beta+\frac{3}{20},\frac{9}{4}\tilde \beta+\frac{1}{10},\frac{9}{4}\tilde \beta+\frac{3}{40}\}>\frac{1}{10}(1+\beta)^3$, we have
\begin{align*}
D(\rho_{k+1})\leq \rho_{k+1}^{\frac{1}{10}}
\end{align*}
by taking $\epsilon'$ sufficiently small, but independent of $k$.

Now for any $\rho \in (0,\rho_0/4)$, we can find a positive integer $k$ such that that $\rho_{k+1}\leq \rho < \rho_k$. Therefore,
\begin{align*}
&\quad A(\rho)+E(\rho)+C^{2/3}(\rho)+D(\rho)\\
&\le \rho_k^3 \rho_{k+1}^{-3}\big(A(\rho_k)+E(\rho_k)+C^{2/3}(\rho_k)+D(\rho_k)\big)\\
&\le 2\rho_k^{\frac 1 {10}-3\beta}\leq 2 \rho^{\frac{1}{1+\beta}(\frac 1 {10}-3\beta)}.
\end{align*}
By choosing $\alpha_0=\frac{1}{1+\beta}(\frac 1 {10}-3\beta)$, the lemma is proved.
\end{proof}

\subsection{Step 3}
\label{boot}
In the final step, we shall use a bootstrap argument to successively improve the decay estimate \eqref{cdf}. However, as we will show below, the bootstrap argument itself only gives the decay of $E_u(\rho)$ no more than $\rho^{2}$, which is not enough for the H\"{o}lder regularity of $u$ since the spatial dimension is four (so that we need the decay exponent $2+\delta$ according to Campanato's characterization of H\"older continuous functions). We shall use parabolic regularity to fill in this gap.

First we prove Theorem \ref{th2}. We begin with the bootstrap argument. We will choose an increasing sequence of real numbers $\{\alpha_k\}_{k=1}^{m}\in (\alpha_0,2)$.

Under the condition \eqref{cond_cdf}, we claim that the following estimates hold uniformly for all $\rho >0$ sufficiently small and $z_1 \in Q(z_0,\rho_0/4)\cap (\partial\Omega\times(t_0-\rho_0^2/16,t_0))$ over the range of $\{\alpha_k\}_{k=0}^m$:
\begin{align}
\label{esti}
&A(\rho,z_1)+E(\rho,z_1)\leq N \rho^{\alpha_k},\,\, C(\rho,z_1)\leq N\rho^{3\alpha_k/2},\\
\label{esti_D}
&D(\rho,z_1)\leq N\rho^{5\alpha_k/6}.
\end{align}
We prove this via iteration. The $k=0$ case for \eqref{esti} and \eqref{esti_D} was proved in \eqref{cdf} with a possibly different exponent $\alpha_0$.
Now suppose that \eqref{esti} and \eqref{esti_D} hold with the exponent $\alpha_k$. We first estimate $A(\rho,z_1)$ and $E(\rho,z_1)$. Let $\rho=\tilde{\gamma}\tilde{\rho}$ where $\tilde{\gamma}=\rho^{\mu}$, $\tilde{\rho}=\rho^{1-\mu}$ and $\mu \in (0,1)$ to be determined. We use \eqref{ers}, and \eqref{esti} to obtain
\begin{equation*}
A(\rho)+E(\rho) \leq N \rho^{2\mu}\rho^{\alpha_k(1-\mu)}+N \rho^{-3\mu}\rho^{\frac{19}{18}\alpha_k(1-\mu)}.
\end{equation*}
Choose $\mu=\dfrac{\alpha_k}{90+\alpha_k}$. Then \eqref{esti} is proved for $A(\rho)+E(\rho)$ with the exponent
\begin{align*}
\alpha_{k+1}&:=\min\Big\{2\mu+\alpha_k(1-\mu),
\frac{19}{18}\alpha_k(1-\mu)-3\mu\Big\}\\
&=\dfrac{92}{90+\alpha_k}\alpha_k \in (\alpha_k,2).
\end{align*}
The estimate in \eqref{esti} with $\alpha_{k+1}$ in place of $\alpha_k$ for $C(\rho,z_1)$ follows from \eqref{hhh} immediately.
To prove the estimate in \eqref{esti_D} with $\alpha_{k+1}$ we will use \eqref{D_est}. Let $\rho=\tilde{\gamma}\tilde{\rho}$, where $\tilde{\gamma}=\rho^{\mu}$ and $\tilde{\rho}=\rho^{1-\mu}$ with a constant $\mu \in (0,1)$ to be specified. From \eqref{D_est}, \eqref{esti} with $\alpha_{k+1}$ in place of $\alpha_k$,and \eqref{esti_D}, we have
\begin{align*}
D(\rho) &\leq N\big[\rho^{-3\mu+\frac{3}{2}\alpha_{k+1}(1-\mu)}+\rho^{9\mu/4+\frac{5}{6}\alpha_{k}(1-\mu)}+\rho^{9\mu/4+\frac{3}{4}\alpha_{k+1}(1-\mu)}\big].
\end{align*}
Choose $\mu=\dfrac{\alpha_{k+1}}{7+\alpha_{k+1}}$. Then we can get
\begin{align*}
&\min\{-3\mu+\frac{3}{2}\alpha_{k+1}(1-\mu),9\mu/4+\frac{5}{6}\alpha_{k}(1-\mu),9\mu/4+\frac{3}{4}\alpha_{k+1}(1-\mu)\}\\&=\frac{15\alpha_{k+1}}{14+2\alpha_{k+1}},
\end{align*}
and
\begin{align*}
D(\rho)\leq N\rho^{\frac{15\alpha_{k+1}}{14+2\alpha_{k+1}}}\leq N\rho^{5\alpha_{k+1}/6}
\end{align*}
since $\alpha_{k+1}\in(0,2)$.
Moreover,
\begin{equation*}
2-\alpha_{k+1}=\dfrac{90}{90+\alpha_k}(2-\alpha_k) \leq \dfrac{90}{90+\alpha_0}(2-\alpha_k).
\end{equation*}
Thus, for any sufficiently small $\delta$, we can find a $m$ that $\alpha_m\geq 2-\delta$.

Via the bootstrap argument, we have got the following estimates for all $\rho >0$ sufficiently small and $z_1=(x_1,t_1) \in Q(z_0,\rho_0/4)\cap (\partial\Omega\times(t_0-\rho_0^2/16,t_0))$:
\begin{equation}
\sup_{t_1-\rho^2 \leq t\leq t_1}\int_{B^{+}(x_1,\rho)}|u(x,t)|^2+|H(x,t)|^2\,dx \leq N \rho^{4-\delta},
\label{pp}
\end{equation}
\begin{equation}
\int_{Q^{+}(z_1,\rho)}|\Pi-[\Pi]_{x_1,\rho}|^{3/2} \,dz \leq N \rho^{3+\frac{5}{6}(2-\delta)},
\label{ppp}
\end{equation}
\begin{equation}
\int_{Q^{+}(z_1,\rho)}|u|^3+|H|^3 \,dz \leq N \rho^{3+\frac{3}{2}(2-\delta)}.
\label{pppp}
\end{equation}

Now we rewrite \eqref{ns} (in the weak sense) into
\begin{equation}
\begin{aligned}
\label{eq11.55}
&\partial_t u_i- \Delta u_i=-\partial_j(u_iu_j)-\partial_i \Pi+\partial_j(H_iH_j),\\
&\partial_t H_i- \Delta H_i=-\partial_j(u_jH_i)+\partial_j(u_iH_j).
\end{aligned}
\end{equation}

Finally, we use the parabolic regularity theory to improve the decay estimate of mean oscillations of $u$ and then complete the proof.
Due to \eqref{pp} and \eqref{pppp}, there exists $\rho_1 \in (\rho/2,\rho)$ such that
\begin{equation}
\begin{aligned}
\int_{B^{+}(x_1,\rho_1)}|u(x,t_1-\rho_1^2)|^2+|H(x,t_1-\rho_1^2)|^2\,dx &\leq N \rho^{4-\delta},\\
\,\,\int_{t_1-\rho_1^2}^{t_1}\int_{S^{+}(x_1,\rho_1)}|u|^3+|H|^3 \,dx\,dt &\leq N \rho^{2+\frac{3}{2}(2-\delta)}.
\end{aligned}
\label{oo}
\end{equation}
Let $v, h$ be the unique weak solution to the heat equation
\begin{align*}
\partial_t v-\Delta v=0 \ \ \text{in} \ \ Q^{+}(z_1,\rho_1),\\
\partial_t h-\Delta h=0 \ \ \text{in} \ \ Q^{+}(z_1,\rho_1),
\end{align*}
with the boundary condition $v=u, h=H$ on $\partial_p Q^{+}(z_1,\rho_1)$. Since $v=0$, $h\cdot \nu=0, (\nabla\times h)\times h=0$ on the flat boundary part, it follows from the standard estimates for the heat equation, H\"{o}lder's inequality, and \eqref{oo} that
\begin{align}
 \sup_{Q^{+}(z_1,\rho_1/2)}(|\nabla v|&+|\nabla h|) \leq N \rho_1^{-6}\int_{t_1-\rho_1^2}^{t_1}\int_{S^{+}(x_1,\rho_1)}(|v|+|h|)\,  dx\,dt\nonumber\\
 &+N\rho_1^{-5}\int_{B^{+}(x_1,\rho_1)}(|v(x,t_1-\rho_1^2)|+|h(x,t_1-\rho_1^2)|)\,dx\nonumber\\
 &\leq N \rho^{-1-\delta/2}.
\label{uu}
\end{align}

Denote $w=u-v, \tilde{h}=H-h.$ Then $w,\tilde{h}$ satisfies the linear parabolic equation
\begin{equation*}
\begin{aligned}
&\partial_t w_i -\Delta w_i = -\partial_j(u_iu_j)-\partial_i (\Pi-[\Pi]_{x_1,\rho})+\partial_j(H_iH_j),\\
&\partial_t \tilde{h}_i -\Delta \tilde{h}_i = -\partial_j(H_iu_j)+\partial_j(u_iH_j).
\end{aligned}
\end{equation*}
with the zero boundary condition. By the classical $L_p$ estimate for  parabolic equations, we have
\begin{align*}
\nonumber\|\nabla w\|_{L_{3/2}(Q^{+}(z_1,\rho_1))} &\leq N \||u|^2\|_{L_{3/2}(Q^{+}(z_1,\rho_1))}+N \||H|^2\|_{L_{3/2}(Q^{+}(z_1,\rho_1))} \\&\quad+N \|\Pi-[\Pi]_{x_1,\rho}\|_{L_{3/2}(Q^{+}(z_1,\rho_1))},
\end{align*}
and
\begin{align*}
\nonumber\|\nabla \tilde{h}\|_{L_{3/2}(Q^{+}(z_1,\rho_1))} &\leq N \||u||h|\|_{L_{3/2}(Q^{+}(z_1,\rho_1))},
\end{align*}
which together with \eqref{ppp}, \eqref{pppp}, and the condition $f\in L_{6}$ yields
\begin{equation}
                        \label{eq10.42}
                        \begin{aligned}
&\int_{Q^{+}(z_1,\rho_1)}|\nabla w|^{3/2}\,dz\le N\rho^{3+5(2-\delta)/6},\\
&\int_{Q^{+}(z_1,\rho_1)}|\nabla \tilde{h}|^{3/2}\,dz\le N\rho^{3+3(2-\delta)/2}.
\end{aligned}
\end{equation}
Since $|\nabla u|\leq |\nabla w|+|\nabla v|$, $|\nabla H|\leq |\nabla h|+|\nabla \tilde{h}|$, we combine \eqref{uu} and \eqref{eq10.42} to obtain, for any $r \in (0,\rho/4)$, that
\begin{equation*}
\begin{aligned}
&\int_{Q^{+}(z_1,r)}|\nabla u|^{3/2} \,dz \leq N \rho^{3+5(2-\delta)/6}+ r^{6}\rho^{-3/2-\frac{3}{4}\delta},\\
&\int_{Q^{+}(z_1,r)}|\nabla H|^{3/2} \,dz \leq N \rho^{3+3(2-\delta)/2}+ r^{6}\rho^{-3/2-\frac{3}{4}\delta}.
\end{aligned}
\end{equation*}
Upon taking $\delta = \frac{1}{20}$ and $r=\rho^{1000/973}/4$ (with $\rho$ small), we deduce
\begin{equation}
                                \label{eq11.50}
\int_{Q^{+}(z_1,r)}|\nabla u|^{3/2}+|\nabla H|^{3/2}\,dz \leq N r^{q},
\end{equation}
where
\begin{equation*}q=\frac{36001}{8000}>6-\frac 3 2.
\end{equation*}

Since $u\in \cH^1_{3/2}$ is a weak solution to \eqref{eq11.55}, it then follows from Lemma \ref{lem11.31}, \eqref{eq11.50}, \eqref{ppp} and \eqref{pppp} with $r$ in place of $\rho$ that
\begin{align*}
\nonumber &\quad\int_{Q^{+}(z_1,r)}|u-(u)_{z_1,r}|^{3/2}\,dz \\
\nonumber &\leq N r^{3/2}\int_{Q^{+}(z_1,r)}\big|\nabla u|^{3/2}+(|u|^2)^{3/2}+(|H|^2)^{3/2}+|\Pi-[\Pi]_{x_1,r}|^{3/2}\big)\,dz \\
\nonumber &\leq N r^{q+3/2}.
\end{align*}
and
\begin{align*}
\nonumber &\quad\int_{Q^{+}(z_1,r)}|H-(H)_{z_1,r}|^{3/2}\,dz \\
\nonumber &\leq N r^{3/2}\int_{Q^{+}(z_1,r)}\big|\nabla H|^{3/2}+(|u||H|)^{3/2}\big)\,dz \\
\nonumber &\leq N r^{q+3/2}.
\end{align*}

We also can prove the above inequality for the interior point under the assumption of Proposition \ref{prop2} by our method. Then
by Campanato's characterization of H\"older
continuous functions near a flat boundary (see, for instance, \cite[Lemma 4.11]{Lieb_96}), that $u, H$ are H\"{o}lder continuous in a neighborhood of $z_0$. This completes the proof of Theorem \ref{th2}.

Theorem \ref{lpq} then follows from Theorem \ref{th2} by applying Proposition \ref{prop3} and Theorem \ref{mainthm} then follows from Theorem \ref{th2} by applying Proposition \ref{prop1}. Finally, we can prove that Theorem \ref{lpq}, \ref{mainthm} also holds for a $C^2$ domain similarly by following the argument in \cite{Seregin_06, Mikhailov_11}. Lastly, under the assumption of Remark \ref{rm1}, the same $\epsilon$-regularity criteria can be proved for interior point with a similar method in \cite{Dong_13b}, then Theorem \ref{th3} is deduced by using the standard argument in the geometric measure theory, which is explained for example in \cite{CKN_82}.

\section*{Acknowledgements.} The author would like to thank Professor Hongjie Dong for his directions and stimulating discussions on this topic. The author also wants to thank Dr. Wendong Wang for helpful discussions and comments. X. Gu was sponsored by the China Scholarship Council for one year study at Brown University and was partially supported by the NSFC (grant No. 11171072) and the the  Innovation  Program  of  Shanghai  Municipal  Education  Commission (grant No. 12ZZ012).

\end{document}